\documentclass[12pt, twoside]{article}
\usepackage{amsmath,amsthm,amssymb}
\usepackage{times}
\usepackage{enumerate}
\usepackage{bm}
\usepackage[all]{xy}
\usepackage{mathrsfs}
\usepackage{amscd}
\usepackage{color}
\usepackage{comment}
\def\titlerunning#1{\gdef\titrun{#1}}
\makeatletter
\def\author#1{\gdef\autrun{\def\and{\unskip, }#1}\gdef\@author{#1}}
\def\address#1{{\def\and{\\\hspace*{18pt}}\renewcommand{\thefootnote}{}%
		\footnote {#1}}%
	\markboth{\autrun}{\titrun}}
\makeatother
\def\email#1{e-mail: #1}
\def\subjclass#1{{\renewcommand{\thefootnote}{}%
		\footnote{\emph{Mathematics Subject Classification (2020):} #1}}}
\def\keywords#1{\par\medskip
	\noindent\textbf{Keywords.} #1}

\newtheorem{theorem}{Theorem}[section]
\newtheorem{corollary}[theorem]{Corollary}
\newtheorem{lemma}[theorem]{Lemma}
\newtheorem{proposition}[theorem]{Proposition}
\theoremstyle{definition}
\newtheorem{definition}[theorem]{Definition}
\newtheorem{remark}[theorem]{Remark}

\numberwithin{equation}{section}

\xyoption{all}

\def \N {\mathbb{N}}
\def \C {\mathbb{C}}

\def \a {\alpha }
\def \b {\beta}

\def \de {\delta}
\def \De {\Delta}
\def \la {\lambda}
\def \La {\Lambda}
\def\w {\omega}
\def\Om{\Omega}

\def\pa{\partial}
\def\na {\nabla}
\def\Ga{\Gamma}

\begin{document}
	\baselineskip=17pt
	
	\titlerunning{Vafa--Witten Equations and Conformal Geometry }
	\title{Vafa--Witten Equations and Conformal Geometry}
	
\author{Teng Huang and Pan Zhang }

\date{}

\maketitle
	\address{T. Huang: School of Mathematical Sciences, University of Science and Technology of China; CAS Key Laboratory of Wu Wen-Tsun Mathematics, University of Science  and Technology of China, Hefei, Anhui, 230026, P. R. China; \email{htmath@ustc.edu.cn;htustc@gmail.com}}
\address{P. Zhang: School of Mathematical Sciences, Anhui University, Hefei, Anhui, 230026, People’s Republic of China; \email{panzhang20100@ahu.edu.cn}}

\begin{abstract}
In this article, we establish geometric and analytic constraints imposed by the existence of nontrivial solutions to the Vafa–Witten equations on closed 4-manifolds. Using conformal invariance and refined Bochner-type estimates, we first prove an inequality relating the Yamabe constant \(Y(g)\) to the \(L^2\)-norm of the self-dual Weyl tensor: \(Y(g)\le 2\sqrt{6}\|W_g^+\|_{L^2}\); when \(Y(g)>0\), this yields a topological lower bound \( \int_M |W_g^+|^2 \ge \frac{4}{3}\pi^2(2\chi(M)+3\sigma(M))\). In the equality case, we show that the manifold must be K\"ahler with nonnegative scalar curvature and that the connection is reducible. As an application, for positive Einstein manifolds with \(\operatorname{Ric}=3g\) admitting an irreducible Vafa--Witten solution, we obtain a sharp volume bound and prove the manifold cannot be K\"ahler.  

Through dimensional reduction \(S^1\times N\), we establish a one-to-one correspondence between stable flat connections on a closed 3-manifold \(N\)  and \(S^1\)-invariant Vafa--Witten solutions, which yields a new estimate for the Yamabe constant \(Y(g_{S^1\times N})\le 2\sqrt{6\pi}\big(\int_N |\operatorname{Ric}(g_N)-\frac13 R_{g_N}g_N|^2\big)^{1/2}\). Finally, under a regularity assumption that every anti-self-dual connection in the compactified moduli space is regular, we prove an energy gap: there exists \(\varepsilon(g,P)>0\) such that any Vafa--Witten solution satisfies either \(F_A^+\equiv0\) or\(\|F_A^+\|_{L^2}\ge\varepsilon\).
\end{abstract}


\section{Introduction}
The study of gauge‑theoretic equations on four‑manifolds has been a cornerstone of both differential geometry and mathematical physics. 
In their seminal work, Vafa and Witten \cite{VW94} introduced a set of such equations on a four‑manifold, 
motivated by the investigation of the $S$-duality conjecture within the framework of topologically twisted $\mathcal{N}=4$ supersymmetric Yang--Mills theory. 
The moduli space of solutions to these equations is anticipated to yield potentially novel invariants of the underlying four-manifold.
 From a physical perspective, the resulting partition functions are expected to serve as generating functions for the Euler characteristics 
 of anti-self-dual (ASD) instanton moduli spaces, with the $S$-duality conjecture predicting that these partition functions 
 exhibit remarkable modularity properties.

The Vafa--Witten equations have also garnered significant attention from alternative geometric viewpoints. 
Haydys \cite{Hay15} and Witten \cite{Wit12} have considered these equations in the context of a program aimed at the `categorification' of Khovanov homology. 
In this approach, a five-dimensional gauge theory is utilized, wherein the Vafa--Witten equations arise naturally as a dimensional reduction of the five-dimensional field equations. Consequently, the solution spaces of the Vafa--Witten equations are expected to play a crucial role in the construction of a Casson--Floer type homology theory.

To establish our framework,  let $M$ be a closed, oriented and smooth Riemannian 4-manifold equipped with a Riemannian metric $g$, 
and let $P\rightarrow M$ be a principal $G$-bundle over $M$ with $G$ being a compact Lie group. 
We denote by $\mathcal{A}_{P}$ the space of all connections on $P$, and by $\Om^{2,+}(M,\mathfrak{g}_{P})$ the space of self-dual two-forms taking values in the adjoint bundle $\mathfrak{g}_{P}$ associated to $P$. 
The Vafa--Witten equations are then formulated for a triple  $(A,B,C)\in\mathcal{A}_{P}\times \Om^{2,+}(M,\mathfrak{g}_{P})\times\Om^{0}(M,\mathfrak{g}_{P})$ as follows:
\begin{equation}\label{E6}
\left\{
\begin{split}
&d_{A}C+d_{A}^{+,\ast}B=0,\\
&F_{A}^{+}+\frac{1}{8}[B_{\bullet}B]+\frac{1}{2}[B,C]=0.
\end{split}
\right.
\end{equation}
These equations are invariant under the action of the gauge group, and the Vafa--Witten moduli space is defined as the space of gauge equivalence classes of such solutions \cite{DG22,Gua23,Tan17}. 

On a closed $4$-manifold, it is a standard result that the Vafa--Witten equations (\ref{E6}) are equivalent to the following system:
\begin{equation}\label{E7}
\left\{
\begin{split}
&d_{A}^{+,\ast}B=d_{A}C=0,\\
&F_{A}^{+}+\frac{1}{8}[B_{\bullet}B]+\frac{1}{2}[B,C]=0.\\
\end{split}
\right.
\end{equation}
In the specific case where the gauge group is $G=SU(2)$ or $SO(3)$ and the connection $A$ is irreducible, the condition $d_{A}C=0$ forces  $C=0$. 
Consequently, the system (\ref{E7}) simplifies to the \textit{reduced Vafa--Witten} equations:
\begin{equation}\label{E8}
\left\{ 
\begin{split}
&d^{+,\ast}_{A}B=0,\\
&F_{A}^{+}+\frac{1}{8}[B_{\bullet}B]=0.\\
\end{split}
\right.
\end{equation}

In the remainder of this article, the term ``Vafa--Witten equations" will specifically refer to the reduced system (\ref{E8}). 
Mares conducted an in-depth study of the analytic aspects of these equations in his Ph.D. thesis \cite{Mar10}. A notable challenge in this theory, analogous to the situation encountered in Hitchin's equations \cite{Hit87},  is the lack of a general a priori $L^{2}$-bound on the curvature $F_{A}$ for a connection $A$ satisfying the equations. This absence renders the study of compactness properties for the moduli space of solutions particularly intricate. Taubes \cite{Tau13,Tau14,Tau17} made significant progress in the study of compactness, partially characterizing the behavior of non-convergent sequences; Tanaka \cite{Tan15,Tan19} investigated the case where the curvature sequence has no concentration points as well as the structure of singular sets on Kähler surfaces.

Furthermore, our current understanding of the existence of solutions to the Vafa--Witten equations remains relatively limited, except in the restricted setting of Kähler surfaces. In this case, Tanaka gave a necessary and sufficient condition for the existence of solutions \cite{Tan15}; simultaneously, Tanaka and Thomas \cite{TT20} proved a vanishing theorem: if $\mathrm{deg}(K_M)\leq 0$, then $\b=0$. Subsequently, Chen \cite{Che24} generalized this result to higher-dimensional Kähler manifolds.

Despite notable advances in the study of the Vafa--Witten equations (cf. \cite{Che24,DG22,Gua23,Hua19,Hua22,Mar10,Tan15,Tan17}), a number of fundamental geometric questions still remain open. For instance:
\begin{itemize}
	\item[(i)] \textbf{What are the necessary conditions on $(M,g)$ for the existence of a nontrivial Vafa--Witten solution?}
	\item[(ii)]\textbf{How does the existence of such a solution constrain the conformal geometry of $M$?}
\end{itemize}
These questions are the central focus of this work. The primary objective of this article is to investigate the geometric constraints that a manifold    must satisfy in order to admit non-trivial solutions to the Vafa--Witten equations. Following the spirit of \cite{LeB97}, where the perturbed Seiberg--Witten equations were employed to deduce bounds on the Yamabe invariant and scalar curvature for certain four-manifolds, we adopt a similar approach. A key observation underpinning our work is the conformal invariance of the Vafa--Witten equations with respect to the metric on a four-manifold. By exploiting this invariance, we are able to derive inequalities for the Yamabe invariant directly from the Vafa--Witten equations.

Our approach is inspired by the work of Gursky, Kelleher and Streets \cite{GKS18}, who established a conformally invariant gap theorem for Yang--Mills connections. Their proof relies on a refined Bochner-type estimate applied to the self-dual component of the curvature, employing solutions to a modified Yamabe problem as constructed in \cite{Gur00,GL98}.  Notably, the bound they obtain is optimal, being saturated by the $SU(2)$ ADHM/BPST instanton on the standard four-sphere. In their setting, the resulting inequality involves the $L^{2}$-norm of the self-dual curvature $F_{A}^{+}$ and depends on a structure constant. 

Adapting their analytical framework to the Vafa--Witten equations, we find that the nonlinear structure simplifies considerably. This leads to a sharper inequality and, under the equality case, yields stronger rigidity conclusions than those available in the Yang--Mills setting. Specifically, the conformal invariance and the algebraic structure of the Vafa--Witten equations allow us to prove a topological lower bound  $2\chi+3\sigma\geq\frac{1}{24}Y^{2}(g)$ under a Ricci curvature condition (Corollary \ref{C2}) and a rigidity result forcing the manifold to be Kähler (Theorem \ref{T2})-none of which have analogues in Yang--Mills theory. Finally, via dimensional reduction to three-manifolds, we obtain new geometric constraints for stable flat connections.

\begin{theorem}\label{T1}
Let $(M,g)$ be a closed, smooth, four-dimensional manifold. Suppose $(A,B)$ is a solution to the Vafa-Witten equations with $B\not\equiv0$. 
Then the following hold:

\noindent\textup{(1)} The Yamabe constant $Y(g)$ satisfies the inequality
\begin{equation} \label{E0}
Y(g)\leq 2\sqrt{6}\|W^{+}_{g}\|_{L^{2}},
\end{equation}
where $W^{+}_{g}$ denotes the self-dual Weyl tensor of $g$.

\noindent\textup{(2)} If the Yamabe constant $Y(g)>0$, then the following topological bound holds:
\begin{equation}\label{E10}
\int_{M}|W^{+}_{g}|^{2}\,d\mathrm{vol}_{g}\geq\frac{4}{3}\pi^{2}\bigg{(}2\chi(M)+3\sigma(M)\bigg{)}.
\end{equation}

\noindent\textup{(3)} If $Y(g)=2\sqrt{6}\|W^{+}_{g}\|_{L^{2}}$, then 
\[
F_{A}^{+}=0\quad\text{and}\quad [B_{\bullet}B]=\nabla_{A}B=0.
\]
\end{theorem}
\begin{remark}
Our inequality \eqref{E0} can be viewed as an extention of the conformally invariant gap estimate for Yang--Mills connections established by Gursky--Kelleher--Streets \cite[Theorem 1.1]{GKS18}:
\[Y(g) \le 3\gamma_1 \|F_\nabla^+\|_{L^2} + 2\sqrt{6}\,\|W^+\|_{L^2},\]
where $\gamma_1$ is a constant depending on the structure group of the bundle. In the setting of the Vafa--Witten equations, the self-dual curvature term $F_A^+$ is coupled to the $B$-fieldvia the equation $F_A^+ = -\frac{1}{8}[B_{\bullet} B]$, and the nonlinear structure simplifies so that the curvature contribution is absorbed, yielding the sharper bound \eqref{E0} with no dependence on the bundle structure.
	
The topological lower bound \eqref{E10} was previously obtained by Gursky under different hypotheses: for all metrics with nonnegative Yamabe invariant on manifolds with $b^{+}>0$ \cite[Theorem 1]{Gur98}; and for metrics with positive Yamabe invariant satisfying the harmonic Weyl condition $\delta W^+=0$ \cite[Theorem A]{Gur00}. In contrast, our estimate \eqref{E10} is derived directly from the existence of a nontrivial solution to the Vafa--Witten equations and requires no a priori curvature or topological assumption beyond $Y(g)>0$. 
\end{remark}
Theorem \ref{T1} provides a general inequality linking the Yamabe constant, the Weyl curvature, and the Vafa--Witten solution. Under a stronger control on the Ricci curvature-namely, that its eigenvalues are uniformly close to $\frac{R_{g}}{4}$, we obtain the following topological obstruction.

\begin{corollary}\label{C2}
	Let $(M,g)$ be a closed, four-dimensional manifold whose Ricci curvature satisfies
	$$\mathrm{Ric}(g)\geq\frac{1}{6}R_{g}>0.$$
	Suppose $(A,B)$ is a solution of Vafa--Witten equations on a principal $G$-bundle $P\rightarrow M$ with $B\not\equiv0$, then
	\[2\pi^{2}\bigg{(}2\chi(M)+3\sigma(M)\bigg{)}\geq\frac{1}{24}Y^{2}(g).\]
\end{corollary}

Having established this general inequality, we proceed to examine the limiting case where the bound is saturated. For an irreducible connection on a principal $SU(2)$ or $SO(3)$-bundle,  the kernel of $d_{A}:\Om^{0}(M,\mathfrak{g}_{P})\rightarrow \Om^{1}(M,\mathfrak{g}_{P})$ is trivial (cf. \cite{DK90,FU91}).

\begin{theorem}\label{T2}
Let $(M,g)$ be a closed, smooth, four-manifold. Suppose that $(A,B)$ is a solution to the Vafa--Witten equations with $B\not\equiv0$ and structure group $SU(2)$ or $SO(3)$. Then the equality $Y(g)=2\sqrt{6}\|W^{+}_{g}\|_{L^{2}}$ holds if and only if $(M,g)$ a K\"{a}hler manifold of non-negative scalar curvature and the connection $A$ is reducible.
\end{theorem}

As a direct consequence of this rigidity result, we obtain a strict inequality for irreducible connections.

\begin{corollary}
Under the same assumptions as in Theorem \ref{T2}, if in addition $A$ is irreducible, then
	\[Y(g)<2\sqrt{6}\|W^{+}_{g}\|_{L^{2}}.\]
\end{corollary}

An important topic in the study of Einstein manifolds concerns the admissible values of the Einstein constants and the related question 
of the rigidity of Einstein structures.  Given a (non-Ricci-flat) Einstein metric $g$, one may normalize it such that $Ric(g)=\pm(n-1)g$. 
The natural question then arises: what are the possible values of the volume $vol(g)$? For the case of four-dimensional Einstein manifolds with positive scalar curvature, Gursky provided a complete classification in \cite{Gur00}. 
The following theorem constitutes an application of the estimate (\ref{E10}) to the setting of positive Einstein manifolds.

\begin{theorem}\label{T3}
Let $(M,g)$ be a closed, four-dimensional Einstein manifold of positive scalar curvature, normalized so that $Ric(g)=3g$. Suppose $(A,B)$ is a solution of Vafa-Witten equations on a principal $SU(2)$ or $SO(3)$-bundle $P\rightarrow M$ with $A$ is irreducible and $B\not\equiv0$. 
Then $(M,g)$ is not K\"{a}hler and satisfies \[\frac{2}{9}\pi^{2}\bigg{(}2\chi(M)+3\sigma(M)\bigg{)}>vol(g). \]
\end{theorem}

Beyond the four-dimensional setting, the Vafa--Witten equations also exhibit intimate ties with gauge-theoretic equations on three-manifolds. Following the dimensional reduction approach (cf. \cite{Hua22,Mar10,Wit12}), we consider the product manifold \(M = S^1 \times N\) equipped with the product metric \(g_{M}= dt^2 + g_N\). Under this reduction, a solution of the Vafa--Witten equations on \(M\) that is invariant under the \(S^1\)-action corresponds precisely to a stable flat connection on the three-manifold \(N\). Recall that a pair \((A, \phi)\) on a principal \(G\)-bundle over \(N\) is called a \emph{stable flat connection} if it satisfies the elliptic system
\begin{equation}\label{E11}
\left\{
\begin{split}
&F_{A}-\phi\wedge\phi=0,\\
&d_{A}\phi=d_{A}^{\ast}\phi=0,\\
\end{split}
\right.
\end{equation}
where \(A\) is a connection on \(P \to N\) and \(\phi\) is a \(\mathfrak{g}_P\)-valued \(1\)-form. This system, introduced in \cite{Cor88, GU12}, generalizes the notion of flat connections and shares invariance properties under both real and complex gauge transformations. On a Riemann surface, it reduces to Hitchin's self-duality equations \cite{Hit87}. Via the correspondence established in Proposition \ref{P6}, the study of stable flat connections on three-manifolds becomes a special case of the Vafa--Witten theory on four-manifolds with a circular symmetry.

Exploiting this correspondence, we apply the conformal geometric estimates developed in the previous sections to derive constraints on the Yamabe constant of the product manifold \(S^1 \times N\) in terms of the Ricci curvature of \(g_N\). In particular, we obtain the following inequality:
\begin{theorem}\label{T5}
Let $(N,g_{N})$ be a closed, smooth, $3$-manifold and let $\mathcal{A}=A+\mathrm{i}\phi$ be a stable flat connection with $\phi\not\equiv0$. Let $M:=S^{1}\times N$ with product metric $g_{M}=dt^{2}+g_{N}$. Then the Yamabe constant $Y(g_{M})$ of $(M,g_{M})$ satisfies
	\begin{equation*}
	Y(g_{M})\leq2\sqrt{6\pi}\bigg{(}\int_{N}\left|\mathrm{Ric}(g_{N})-\frac{1}{3}R_{g_{N}}g_{N}\right|^{2}dvol_{g_{N}}\bigg{)}^{\frac{1}{2}}.
	\end{equation*}
In particular, if $g_{N}$ is Einstein, then
$$Y(g_{M})\leq 0.$$
\end{theorem}
Beyond the geometric and topological constraints discussed above, the existence of a nontrivial solution to the Vafa--Witten equations also implies a quantitative analytic gap, which refines the compactness theory for these equations. Our analysis in this direction is inspired by Feehan's work \cite{Fee16} on the gap phenomenon for Yang–Mills connections, and we extend his ideas to the Vafa--Witten setting.

Recall that the compactification of the ASD moduli space often involves concentration phenomena (cf.\cite{DK90,Fee16,Sed82}). In contrast, for convergent sequences of Vafa–Witten solutions, we study the least eigenvalue $\mu(A)$ of the operator $d^{+}_{A}d_{A}^{+,\ast}$ on self-dual $\mathfrak{g}_{P}$-valued two-forms. Our main result in this part shows that if every ASD connection $A_{asd}$ in the compactified moduli space is regular (i.e. $\mu(A_{asd})>0$),  then every Vafa--Witten solution must satisfies either $F_{A}^{+}=0$ or $\|F_{A}^{+}\|_{L^{2}(M)}\geq\varepsilon$ for some positive constant $\varepsilon$. Consequently, a non-ASD solution cannot be arbitrarily close to the ASD locus. For simplicity, denote by $\overline{\mathcal{M}}_{ASD}(P, g)$ the Uhlenbeck compactification of the moduli space of ASD connections.

\begin{theorem}\label{T8}
	Let $(M,g)$ be a closed, smooth, four-dimensional manifold. Suppose that all connections in $\overline{\mathcal{M}}_{ASD}(P, g)$ are \textit{regular}.
	 Then there is a positive constant $\varepsilon=\varepsilon(g,P)$ such that  any smooth solution $(A,B)$ of Vafa--Witten equations satisfies
	$$\|F_{A}^{+}\|_{L^{2}(M)}\geq \varepsilon(g,P)$$
	unless $F_{A}^{+}=0$. 
\end{theorem}

\begin{remark}
It is instructive to compare Theorem \ref{T8} with a known result in the K\"ahler surface setting. Let $M$ be a closed K\"ahler surface and $P\to M$ a principal $SU(2)$-bundle with $c_2(P)=c>0$. Let $g$ be a $c$-generic K\"ahler metric in the sense of \cite[Definition 1.4]{Hua19}, meaning that no reducible ASD connections appear in the Uhlenbeck compactification for any bundle with $c_2\le c$.

Under this $c$-generic assumption, Huang \cite[Theorem 4.18]{Hua19} proved that there exists a constant $C>0$ such that any smooth solution $(A,B)$ of the Vafa--Witten equations on $M$ satisfies either
\[\Lambda_\omega F_A = 0 \quad\text{or}\quad \|\beta\|_{L^2(M)}\ge C,\]
where $\beta$ is the $(2,0)$-part of $B$. This result is analogous in spirit to Theorem \ref{T8}, as both establish a uniform gap away from the ASD locus. 
\end{remark}

\section{Preliminary Inequalities }\label{sec:prelim}

In this section, we establish the preliminary analytical inequalities that are essential for the proof of Theorem \ref{T1}. Our main tools are the Bochner formula, pointwise estimates for the Weyl curvature term, and a sharp Kato inequality for Lie algebra‑valued harmonic two‑forms on four‑manifolds.

We first fix our notational conventions. For sections of $\mathfrak{g}_{P}$ we define the canonical inner product
\begin{equation}\label{E1}
\langle A,B\rangle:=-\frac{1}{2}\mathrm{tr}(AB).
\end{equation}
The inclusion of the factor $\frac{1}{2}$ aligns with the convention established by Bourguignon and Lawson (\cite[(2.14)]{BL81}). This inner product is positive definite. In the statement of the Böchner formula below, all inner products are understood to be defined using the given Riemannian metric in conjunction with the algebraic inner product (\ref{E1}).  We denote by $\De_{g}=dd^{\ast_{g}}+d^{\ast_{g}}d$ the Hodge Laplacian with respect to metric $g$ on $M$.

For local bases one has, for $P,Q\in\Omega^{2,+}(\mathfrak{g}_{P})$, the tensor $[P,Q]\in\Omega^{2,+}(\mathfrak{g}_{P})$ defined via
\begin{equation}\label{E3}
[P,Q]^{\beta}_{ij\alpha}:=g^{kl}\bigl(P^{\beta}_{ik\delta}Q^{\delta}_{jl\alpha}-P^{\delta}_{ik\alpha}Q^{\beta}_{jl\delta}-P^{\beta}_{jk\delta}Q^{\delta}_{il\alpha}+P^{\delta}_{jk\alpha}Q^{\beta}_{il\delta}\bigr),
\end{equation}
and the Weyl curvature action is given by
\begin{equation*}
(W^{+}_{g}\star\omega)^{\beta}_{ij\alpha}:=g^{kp}g^{lq}W^{+}_{ijkl}\omega_{pq\alpha}^{\beta}.
\end{equation*}
A fundamental tool in our analysis is the following Bochner formula for harmonic two‑forms with values in the adjoint bundle.
\begin{lemma}(\cite[Theorem 3.10]{BL81})\label{L1}
Let $(M,g)$ be a Riemannian manifold and let $P\rightarrow M$ be a principal $G$-bundle with connection $A$.  
If $\w\in\Om^{2,+}(\mathfrak{g}_{P})$ is a harmonic two-form, then 
\begin{equation}\label{E2}
-\frac{1}{2}\De_{g}|\w|^{2}=|\na_{A}\w|^{2}+\frac{1}{3}R_{g}|\w|^{2}-2\langle\w,W_{g}^{+}\star \w\rangle-\langle\w,[F^{+}_{A},\w]\rangle.
\end{equation}
\end{lemma}
We now bound the term involving the Weyl curvature. Since the action of $W_{g}^{+}$ on Lie algebra-valued forms is induced naturally from its action on ordinary real-valued self-dual two-forms, the following estimate follows by a straightforward modification of that case.
\begin{lemma}(\cite[Lemma 2.2]{GKS18})\label{L2}
Let $(M,g)$ be a Riemannian four-manifold, and let $P\rightarrow M$ be a principal $G$-bundle with connection $A$. For any $\w\in\Om^{2,+}(\mathfrak{g}_{P})$,
\begin{equation}\label{E4}
|\langle\w,W_{g}^{+}\star \w\rangle|\leq\frac{2}{\sqrt{6}}|W^{+}_{g}||\w|^{2}.
\end{equation}
\end{lemma}
In four dimensions, harmonic two-forms satisfy an improved Kato inequality.
 Such an inequality was proved for Yang--Mills connections on $\mathbb{R}^{4}$ in \cite{Rad93}. 
 The following result was proved for real-valued harmonic two-forms by Seaman \cite{Sea91} using conformal invariance,
  and was extended to the Lie algebra-valued setting in \cite{GKS18} via an elementary adaptation of Seaman's method.
   For completeness we state the sharp form needed in this article.

\begin{proposition}(\cite[Proposition 2.8]{GKS18})\label{P3}
Let $(M,g)$ be a closed, oriented, smooth four-dimensional Riemannian manifold and let $P\rightarrow M$ be a principal $G$-bundle with connection $A$. If $\w\in\Om^{2,+}(\mathfrak{g}_{P})$ is a harmonic two-form, then 
\begin{equation}\label{E5}
|\na_{A}\w|^{2}\geq\frac{3}{2}|\na|\w||^{2}.
\end{equation}
\end{proposition}

\section{Analysis and Rigidity of the Vafa--Witten Equations}\label{sec:main} 

\subsection{Vafa--Witten Equaions}
Let us now specialize the preceding analytical framework to the setting of the Vafa--Witten equations (cf. \cite{Mar10,Tan17}).  Let $(M,g)$ be a closed, oriented, smooth four-manifold and let $P\to M$ be a principal $G$-bundle with compact Lie group $G$. The Hodge star operator on $2$-forms induces the well-known decomposition
 \[\La^{2}(M)=\La^{2,+}(M)\oplus\La^{2,-}(M),\]
into the $\pm$-eigenspaces of $\ast$. Consequently, the space of $\mathfrak{g}_P$-valued $2$-forms splits accordingly:
\[\Om^{2}(M,\mathfrak{g}_{P})=\Om^{2,+}(M,\mathfrak{g}_{P})\oplus\Om^{2,-}(M,\mathfrak{g}_{P}).\]
This decomposition can be understood via the isomorphism $\mathfrak{so}(4)\cong\mathfrak{so}(3)\oplus\mathfrak{so}(3)$, where $\Lambda^{2}T_{x}^{\ast}M$ is identified with the subalgebra of skew-symmetric endomorphisms of $T_{x}M$. In particular, $\Lambda^{2,+}$ inherits a pointwise Lie bracket, denoted $[\cdot,\cdot]_{\Lambda^{2,+}}$. Combining this with the Lie bracket on $\mathfrak{g}_P$, we define a bilinear map
\[ [\cdot_{\bullet}\cdot]:(\La^{2,+}\otimes\mathfrak{g}_{P})\otimes(\La^{2,+}\otimes\mathfrak{g}_{P})\rightarrow\La^{2,+}\otimes\mathfrak{g}_{P}
\]
by
\begin{equation}\label{E17} \frac{1}{2}[\cdot,\cdot]_{\La^{2,+}}\otimes[\cdot,\cdot]_{\mathfrak{g}_{P}}.
\end{equation}
For a detailed discussion we refer the reader to  \cite[Section 2]{Tan17} and \cite[Appendix A]{Mar10} and for details.

Let $A$ be a connection on $P$ and let $B$ be a section of the associated bundle $\Om^{2,+}\otimes\mathfrak{g}_{P}$. The reduced Vafa–Witten equations are 
\begin{equation*}
\left\{
\begin{split}
&F_{A}^{+}+\frac{1}{8}[B_{\bullet}B]=0\\
& d_{A}^{+,\ast}B=0.
\end{split}
\right.
\end{equation*}  
where $F_{A}^{+}$ denotes the self‑dual part of the curvature of $A$, and $[B_{\bullet}B]$ is defined as in (\ref{E17}).

A fundamental property of this system, which is crucial for our analysis, is its behavior under conformal changes of the metric.
	\begin{lemma}\label{L3}(Conformal Invariance)
		Let $(M,g)$ be a closed, smooth, four-manifold. Suppose $(A,B)$ is a solution of the  Vafa-Witten equations with respect to metric $g$. Let $\tilde{g}=u^{2}g$ be any metric conformal to $g$. Then $(A,B)$ also satisfies the Vafa–Witten equations with respect to $\tilde{g}$.
	\end{lemma}
	\begin{proof}
In four dimensions, the Hodge star operator on $2$-forms is conformally invariant
$$\ast_{\tilde{g}}\eta =\ast_{g}\eta.$$
Hence the self-dual $2$-form is invariant, so
\[F_A^{+,\tilde{g}} = F_A^{+,g},\quad B^{\tilde{g}}=B^{g}.\]
Moreover, by the definition of $[\cdot_{\bullet}\cdot]$, we have
\[ [B^{\tilde{g}}_{\bullet}B^{\tilde{g}}]=[B^{g}_{\bullet}B^{g}].\]

For the divergence term we compute
\[d_{A}^{+,\ast_{\tilde{g}}}B^{\tilde{g}}=-\ast_{\tilde{g}}d_{A}\ast_{\tilde{g}}B^{\tilde{g}}=-\ast_{\tilde{g}}d_{A}B^{\tilde{g}}=-\ast_{\tilde{g}}d_{A}B^{g}.\]
Since $(A,B)$ solves the equations with respect to $g$, we have
$$F_{A}^{+,g}+\frac{1}{8}[B^{g}_{\bullet} B^{g}]=0,\quad d_{A}^{+,\ast_{g}}B^{g}=0.$$
The latter implies $d_{A}B^{g}=0$ because $B^{g}$ is self‑dual and $d_{A}^{\ast_{g}}=-\ast_{g}d_{A}\ast_{g}$. Using the conformal invariance of the star operator, we obtain
		\[
		F_{A}^{+,\ast_{\tilde{g}}}+ \frac{1}{8}[B^{\tilde{g}}_{\bullet}B^{\tilde{g}}] =F_{A}^{+,\ast_{g}}+ \frac{1}{8}[B^{g}_{\bullet}B^{g}] = 0,
		\]
		and
		\[
		d_{A}^{+,\ast_{\tilde{g}}}B^{\tilde{g}}=-\ast^{\tilde{g}}d_{A}B^{g}=0.
		\]
		Thus $(A,B)$ solves the Vafa-Witten equations with respect to $\tilde{g}$. 
	\end{proof}

Leveraging the conformal invariance established in Lemma \ref{L3},
 we now derive a Bochner-type identity tailored to solutions of the reduced Vafa--Witten equations \eqref{E8}. 
 To this end, we apply Lemma \ref{L1} to the self-dual form \(\omega = B\). 
A direct computation using the second equation in \eqref{E8} yields
\[\langle B,[F^{+}_{A},B]\rangle=\langle [B_{\bullet}B],F_{A}^{+}\rangle=-8|F^{+}_{A}|^{2}. \]
Substituting this relation into the Bochner formula \eqref{E2} leads to the following identity.

\begin{corollary}(cf. \cite[Section 3.1]{Mar10})\label{C1}
	Let $(M,g)$ be a Riemannian four-manifold, and $P\rightarrow M$ a principal $G$-bundle with connection $A$.
	 If $(A,B)$ is a solution of the Vafa--Witten equations, then $B$ satisfies the identity
	\begin{equation}
	-\frac{1}{2}\De_{g}|B|^{2}=|\na_{A}B|^{2}+\frac{1}{3}R_{g}|B|^{2}-2\langle B,W_{g}^{+}\star B\rangle+8|F_{A}^{+}|^{2}.
	\end{equation}
\end{corollary}

\subsection{A Conformal Invariant Estimate for Vafa--Witten Solutions}
To establish a conformally invariant estimate for solutions of the Vafa--Witten equations, 
we first recall a family of conformally covariant operators introduced by Gursky in the study of four‑manifolds with harmonic self‑dual Weyl curvature 
(cf. \cite{Gur00,Ito05}).  For an  $n$-dimensional Riemannian manifold $(M,g)$ and a real parameter  $t\in\mathbb{R}$, we define the operator 
\begin{equation}
\mathcal{L}^{t}_{g}=\frac{4(n-1)}{(n-2)}\De_{g}+R_{g}-t|W_{g}|,
\end{equation}
where $|W_{g}|$ denotes the norm of the Weyl curvature of $g$. In the specific case where $n=4$, which is our primary focus, this operator reduces to
\begin{equation}
\mathcal{L}^{t}_{g}=6\De_{g}+R_{g}-t|W^{\pm}_{g}|.
\end{equation}
The utility of this operator stems from its conformal covariance property.
\begin{lemma}(\cite[Lemma 3.1]{Gur00})
Let $\tilde{g}=u^{4/(n-2)}g$ be a metric conformal to $g$. Then the operator transforms according to
\begin{equation}
\mathcal{L}^{t}_{\tilde{g}}\phi=u^{-\frac{n+2}{n-2} }\mathcal{L}^{t}_{g}(\phi u).
\end{equation}
\end{lemma}
For our purposes, we focus on the curvature invariant
\begin{equation}
F_{g}^{\pm}=R_{g}-t|W^{(\pm)}_{g}|.
\end{equation}
The following proposition guarantees the existence of metrics with signed definite $F_{\tilde{g}}^{\pm}$ within a given conformal class.

\begin{proposition}(\cite[Proposition 3.2]{Gur00})\label{P1}
Every conformal class on a closed manifold $(M,g)$ admits a $C^{2,\a}$ metric $\tilde{g}=u^{4/(n-2)}g$ with either $F_{\tilde{g}}^{\pm}>0$, $F_{\tilde{g}}^{\pm}<0$, or $F_{\tilde{g}}^{\pm}=0$; moreover, these three possibilities are mutually exclusive.	
\end{proposition}
For the remainder of this section, we restrict our attention to four dimensions and study the operator 
\[\mathcal{L}_{g}=6\De_{g}+R_{g}-2\sqrt{6}|W^{+}_{g}|  \]
and the corresponding curvature invariant 
$$F^{+}_{g}=R_{g}-2\sqrt{6}|W^{+}_{g}|.$$
In analogy with the Yamabe problem, let us define the functional
\[\hat{Y}[u]=\langle u,\mathcal{L}_{g}u\rangle/\|u\|^{2}_{L^{4}} \]
and the conformal invariant 
\[\hat{Y}(g)=\inf_{u\in L^{2}_{1}}\hat{Y}[u]. \]
The relationship between this new invariant and the standard Yamabe constant is clarified by Gursky \cite[Proposition 3.5]{Gur00},
 and the proof relies on solving a minimisation problem. 

\begin{proposition}(\cite[Proposition 3.5]{Gur00})\label{P2}
If $Y(g)>0$ and $\hat{Y}(g)\leq 0$, then there is a smooth metric $\tilde{g}=u^{2}g$ such that
\[\int_{M}R^{2}_{\tilde{g}}dvol_{\tilde{g}}\leq 24\int_{M}|W^{+}_{\tilde{g}}|^{2}dvol_{\tilde{g}}. \]
\end{proposition}

The following proposition links the existence of nontrivial Vafa--Witten solutions to the sign of $\hat{Y}(g)$. 
Its proof follows the Bochner technique developed in \cite{GKS18} for Yang--Mills connections, adapted here to the Vafa--Witten setting.

\begin{proposition}\label{P4}
If $(A,B)$ is a solution to Vafa--Witten equations with $B\not\equiv 0$ on $(M,g)$, then $\hat{Y}(g)\leq 0$.
\end{proposition}

\begin{proof}
Starting from Corollary \ref{C1} and applying the pointwise estimate from Lemma \ref{L2}, we immediately obtain
\begin{equation*}
\begin{split}
-\frac{1}{2}\De_{g}|B|^{2}&\geq|\na_{A}B|^{2}+\frac{1}{3}(R-2\sqrt{6}|W^{+}|)|B|^{2}+8|F^{+}_{A}|^{2}\\
&\geq|\na_{A}B|^{2}+\frac{1}{3}(R-2\sqrt{6}|W^{+}|)|B|^{2}.\\
\end{split}
\end{equation*}
Away from the zero locus of $B$, the Leibniz rule yields
\[\frac{1}{2}\De_{g}|B|^{2}=|B|\De_{g}(|B|)-|\na|B||^{2}.\]
Substituting this into the previous inequality and employing the sharp Kato inequality from Proposition \ref{P3}, we find
\begin{equation*}
\begin{split}
-|B|\De_{g}(|B|)&\geq|\na_{A}B|^{2}+\frac{1}{3}(R-2\sqrt{6}|W^{+}|)|B|^{2}-|\na|B||^{2}\\
&\geq \frac{1}{2}|\na|B||^{2}+\frac{1}{3}(R-2\sqrt{6}|W^{+}|)|B|^{2},\\
\end{split}
\end{equation*}
which implies
\[-\De_{g}|B|\geq\frac{1}{2}\frac{|\na|B||^{2} }{|B|}+\frac{1}{3} (R-2\sqrt{6}|W^{+}|)|B|.\]
It follows that 
\[-\De_{d}|B|^{\frac{1}{2}}\geq\frac{1}{6} (R-2\sqrt{6}|W^{+}|)|B|^{\frac{1}{2}}.\]
This inequality implies that the first eigenvalue of $\mathcal{L}_{g}$ is non-positive; hence, $\hat{Y}(g)\leq0$.
\end{proof}
We now assemble these preliminary results to complete the proof of our main theorem.

\begin{proof}[\textbf{Proof of Theorem \ref{T1}}]
Let $(A,B)$ be a solution of the Vafa--Witten equations. By Lemma \ref{L3}, the equations are invariant under a conformal change of metric $\tilde{g}=u^{2}g$. 

Assume, for contradiction, that there exists a metric $\tilde{g}$ on $M$ conformal to $g$ for which the curvature invariant satisfies $F^{+}_{\tilde{g}}>0$. Then applying the Böchner identity to the solution $B$ with respect ot the metric $\tilde{g}$, we obtain
\begin{equation*}
	\begin{split}
	-\frac{1}{2}\De_{\tilde{g}}|B|^{2}
	&=|\na_{A}B|^{2}+\frac{1}{3}R_{\tilde{g}}|B|^{2}-2\langle B,W^{+}_{\tilde{g}}\star B\rangle+\frac{1}{8}|[B_{\bullet}B]|^{2}\\
	&\geq |\na_{A}B|^{2}+8|F_{A}^{+}|^{2}+\frac{1}{3}F_{\tilde{g}}|B|^{2}.
	\end{split}
	\end{equation*}
Integrating this inequality over the closed manifold $M$ yields
	\[0\geq \int_{M}(|\na_{A}B|^{2}+8|F_{A}^{+}|^{2}+\frac{1}{3}F_{\tilde{g}}|B|^{2})dvol_{\tilde{g}}. \]
Since $F_{\tilde{g}}^{+}>0$ by assumption, this forces $B=0$ and $F_{A}^{+}=0$ everywhere, contradicting the hypothesis $B\neq0$.
Therefore no metric in the conformal class can have $F_{\tilde{g}}^{+}>0$. By Proposition \ref{P1}, this implies that for any metric $\tilde{g}$ in the conformal class $[g]$, we must have $F_{\tilde{g}}^{+}\leq0$.
The role of the condtion $F_{\tilde{g}}^{+}\leq0$ is primarily to eliminate the hypothetical “positive” branch from consideration, 
ensuring that the minimisation problem in Proposition \ref{P2} is carried out in the correct sign regime.  

Since $(M,g)$ admits a $C^{2,\a}$ metric $\tilde{g}=u^{2}g$ with $F_{\tilde{g}}^{+}<0$, we have
\[\int_{M}R_{\tilde{g}}dvol_{\tilde{g}}\leq 2\sqrt{6}\int_{M}|W^{+}_{\tilde{g}}|dvol_{\tilde{g}}\leq 2\sqrt{6}\Big(\int_{M}|W^{+}_{\tilde{g}}|^{2}dvol_{\tilde{g}}\Big)^{\frac{1}{2}}\sqrt{\mathrm{Vol}(\tilde{g})}.\]
By definition of Yamabe constant of $(M,g)$, we then have
\[Y(g)=\inf_{\tilde{g}\in [g]}\frac{\int_{M}R_{\tilde{g}}dvol_{\tilde{g}} }{\sqrt{\mathrm{Vol}(\tilde{g})}}\leq 2\sqrt{6}\Big(\int_{M}|W^{+}_{\tilde{g}}|^{2}dvol_{\tilde{g}}\Big)^{\frac{1}{2}} . \]
In four dimensions, $\int_{M}|W^{+}_{\tilde{g}}|^{2}dvol_{\tilde{g}}$ is conformally invariant; consequently, 
$$Y(g)\leq2\sqrt{6}\|W^{+}_{g}\|_{L^{2}(M)}.$$
Now assume in addition that the Yamabe constant is positive, i.e.  $Y(g)>0$. Propositions \ref{P2} and \ref{P4} together imply the existence of a metric $\tilde{g}$ conformal to $g$ satisfying the integral inequality
\begin{equation}\label{E12}
Y^{2}(g)\leq \int_{M}R^{2}_{\tilde{g}}dvol_{\tilde{g}}\leq 24\int_{M}|W^{+}_{\tilde{g}}|^{2}dvol_{\tilde{g}},
\end{equation}
which implies $\eqref{E0}$.

Using the Chern--Gauss--Bonnet and signature formulas,
\begin{equation}\label{E18}
\begin{split}
2\pi^{2}(2\chi(M)+3\sigma(M))&=\int_{M}|W^{+}_{\tilde{g}}|^{2}dvol_{\tilde{g}}+\frac{1}{48}\int_{M}|R_{\tilde{g}}|^{2}dvol_{\tilde{g}}\\
&-\frac{1}{4}\int_{M}|E_{\tilde{g}}|^{2}dvol_{\tilde{g}},
\end{split}
\end{equation}
where $E_{\tilde{g}}$ denotes the trace-free Ricci tensor of $\tilde{g}$. Substituting the estimate (\ref{E12}) into (\ref{E18}) yields
\begin{equation*}
2\pi^{2}(2\chi(M)+3\sigma(M))\leq\frac{3}{2}\int_{M}|W^{+}_{\tilde{g}}|^{2}dvol_{\tilde{g}}-\frac{1}{4}\int_{M}|E_{\tilde{g}}|^{2}dvol_{\tilde{g}}.
\end{equation*}
Since the self-dual Weyl tensor is conformally invariant, rearranging gives the desired topological bound
\[\int_{M}|W^{+}_{g}|^{2}=\int_{M}|W^{+}_{\tilde{g}}|^{2}\geq\frac{4}{3}\pi^{2}(2\chi(M)+3\sigma(M)),\]
which is exactly inequality $\eqref{E10}$.

Finally, if equality  $Y(g)=2\sqrt{6}\|W^{+}_{g}\|_{L^{2}}$ holds, then we must have
\[0=\int_{M}-\frac{1}{2}\De_{\tilde{g}}|B|^{2}dvol_{\tilde{g}}\geq\int_{M}(|\na_{A}B|^{2}+8|F_{A}^{+}|^{2})dvol_{\tilde{g}}. \]
which forces $F^{+}_{A}=0$ and $\na_{A}B=0$. This completes the proof of Theorem \ref{T1}.

\end{proof}
The following lemma will be useful in proving Corollary \ref{C2}.
\begin{lemma}\label{L5}
Let $4s=a_{1}+a_{2}+a_{3}+a_{4}$ and $a=a_{1}\leq a_{2}\leq a_{3}\leq a_{4}$. Then
\[ \sum_{i=1}^{4}(a_{i}-s)^{2}\leq 12(s-a)^{2}.\]
In particular, equality holds if and only if $a_{1}=a_{2}=a_{3}=a$ and $a_{4}=4s-3a$. 
\end{lemma}

\begin{proof}
Set $b_{i}=a_{i}-s$, then $\sum_{i=1}^{4}b_{i}=0$ and $a-s=b_{1}\leq\cdots\leq b_{4}$. Therefore, 
\begin{equation*}
\begin{split}
\sum_{i=1}^{4}(a_{i}-s)^{2}&=\sum_{i=1}^{4}b_{i}^{2}=(a-s)^{2}+b_{2}^{2}+b_{3}^{2}+b^{2}_{4}\\
&=(a-s)^{2}+b_{2}^{2}+\big{(}(s-a-b_{4})-b_{2}\big{)}^{2}+b_{4}^{2}.
\end{split}
\end{equation*}
For fixed $b_{4}$, noting that $s-a-b_{4}\geq-2(a-s)$, then the quadratic expression $b_2^2 + ((s-a-b_4)-b_2)^2$ is maximized when $b_2 =a-s$, yielding
$$b_{2}^{2}+\big{(}(s-a-b_{4})-b_{2}\big{)}^{2}\leq(a-s)^{2}+(2(s-a)-b_{4})^{2}.$$ 
Consequently,
\[\sum b_i^2 \le 2(s-a)^{2}+\big(2(s-a)-b_4\big)^2 + b_4^2.
\]
The quadratic expression $\big(2(s-a)-b_4\big)^2 + b_4^2$ is maximized when $b_3=a-s$ and $b_4=3(s-a)$.
Hence
\[
\sum b_i^2 \le 12(s-a)^2,
\]
equality holds if and only if $b_{1}=b_{2}=b_{3}=a-s$ and $b_{4}=3(s-a)$. 
\end{proof}

\begin{proof}[\textbf{Proof of Corollary \ref{C2}}]
Let $\{e_{1},\cdots,e_{4}\}$ be an orthonormal basis of $T_{x}M$ diagonalizing $\mathrm{Ric}(g)$ at a point $x\in M$
 with corresponding eigenvalues $\la_{1}\leq\cdots\leq\la_{4}$. Then 
$$\la_{1}+\dots+\la_{4}=R_{g},$$ 
where $R_{g}$ denotes the scalar curvature.	
	
Under the condition $\la_{1}\geq\frac{1}{6}R_{g}$ and Lemma \ref{L5}, we obtain
\[|E_{g}|^{2}=|\mathrm{Ric}(g)-\frac{R_{g}}{4}g|^{2}=\sum_{i=1}^{4}|\la_{i}-\frac{R_{g}}{4}|^{2}\leq \frac{1}{12}R^{2}_{g}.\]
Using the Chern--Gauss--Bonnet and signature formulas (see (\ref{E18})), we obtain
\begin{equation*}
\begin{split}
\int_{M}|W^{+}_{g}|^{2}dvol_{g}&=2\pi^{2}(2\chi(M)+3\sigma(M))+\frac{1}{4}\int_{M}|E_{g}|^{2}dvol_{g}-\frac{1}{48}\int_{M}|R_{g}|^{2}dvol_{g}\\
&\leq 2\pi^{2}(2\chi(M)+3\sigma(M)).
\end{split}
\end{equation*}
Applying Theorem \ref{T1}, we get
\begin{equation*}
2\pi^{2}(2\chi(M)+3\sigma(M))\geq\frac{1}{24}Y^{2}(g)>0.
\end{equation*}
Equality holds if and only if \(Y^{2}(g)=24\|W^{+}_{\tilde{g}}\|^{2}_{L^{2}}=12\int_{M}|E_{\tilde{g}}|^{2}dvol_{\tilde{g}}\).
\end{proof}

\subsection{Rigidity in the Equality Case}
The results of the previous section show that the equality case in Theorem \ref{T1} forces the solution $(A,B)$ to satisfy $F_{A}^{+}=0$, $[B_{\bullet}B]=\na_{A}B=0$. In this section we analyze the geometric consequences of these conditions, leading to a characterization of the equality case in Theorem \ref{T2}. Throughout this section we assume that the structure group of the principal bundle $P$ is either $SU(2)$ or $SO(3)$.

We first recall the notion of rank of a section $B\in\Om^{2,+}(M,\mathfrak{g}_{P})$ (see \cite{DG22} or \cite[Definiton 1.5]{Tan17}). Let $d=\dim G$. Choose local frames for $\mathfrak{g}_{P}$ and $\La^{2,+}(T^{\ast}M)$; note that $\dim\La^{2,+}(T^{\ast}M)=3$. With respect to these local frames, the section $B$ is represented by a $d\times 3$ matrix-valued function. The rank of $B$ at a point $x\in M$ is the rank of the matrix at $x$. We denote by ${\rm{rank}}(B)$ the maximum of the pointwise rank over $M$. The pointwise rank of $B$ also gives a stratification of the manifold $M$, namely,
\begin{equation}
M^{i}(B) =\{x\in M:{\rm{rank}}(B(x))=i\},\quad 0\leq i\leq {\rm{rank}}(B).
\end{equation}
The top rank stratum is a nonempty open subset of $M$.  If the structure group of the principal bundle $P$ is either $SU(2)$ or $SO(3)$, then the possibilities for the rank of $B$ are less than or equal to $3$. 

The condition $[B_{\bullet}B]=0$ implies that the components of $B$ are linearly dependent at each point, 
which meaning that $B$ has at most rank one (see \cite[Lemma 3.5]{Hua24} or \cite[Lemma 1.6]{Tan17}). 
The following lemma shows that the conditions $[B_{\bullet}B]=0$ and $\na_{A}B=0$ force $B$ to have rank one unless it vanishes identically. 

\begin{lemma}\label{L4}
Let $P\rightarrow M$ be a principal $SU(2)$ or $SO(3)$ bundle over a closed four-dimensional Riemannian manifold $M$. If $B\in\Om^{+}(M,\mathfrak{g}_{P})$ satisfies
$$[B_{\bullet}B]=0\quad and\quad \na_{A}B=0,$$ 
then the rank of $\mathrm{rank}(B)=1$ unless $B\equiv0$.
\end{lemma}

\begin{proof}
According to \cite[Lemma 1.6]{Tan17}, the commutator condition $[B_{\bullet}B] = 0$ forces the pointwise rank of $B$ to satisfy $\mathrm{rank}(B) \leq 1$. Consequently, we can define the zero locus of $B$ as
\[M^{0}(B)=\{x\in M: \mathrm{rank}(B(x))=0\}=\{x\in M:B(x)=0 \}.\]
Suppose, for the sake of contradiction, that there exists a point $x \in M$ where $B(x) \neq 0$. Because $\nabla_{A} B = 0$, the covariant derivative of $B$ vanishes identically, which immediately implies that its pointwise norm $|B|$ is non-zero constant. As a result, the zero set $M^{0}(B)$ is empty. It follows directly that $\mathrm{rank}(B) = 1$ at every point of $M$.
\end{proof}
We now recall the definition of an irreducible connection on a principal $G$-bundle (cf. \cite{DK90,Hua19,Tan17}). For a connection $A$ on a principal $G$-bundle $P\rightarrow M$, its stabilizer $\Ga_{A}$ of $A$ in the gauge group  $\mathcal{G}_{P}$ is defined as
$$\Ga_{A}:=\{u\in\mathcal{G}_{P}|u^{\ast}(A)=A \}.$$
A connection $A$ is called reducible if its stabilizer $\Ga_{A}$ is strictly larger than the center $C(G)$ of $G$. Otherwise, $A$ said to be irreducible, and in this case
$$\Ga_{A}\cong C(G).$$
For $G=SU(2)$ or $SO(3)$, a connection $A$ is irreducible if and only if it admits no nontrivial covariantly constant $\mathfrak{g}_{P}$-value $0$-form; equivalently,
$$\ker d_{A}\cap \Om^{0}(M,\mathfrak{g}_{P})=\{0\}.$$
We now proceed to characterize the equality case in Theorem \ref{T2}.
\begin{proof}[\textbf{Proof of Theorem \ref{T2}}]
Assume that the equality $Y(g)=2\sqrt{6}\|W^{+}_{g}\|_{L^{2}}$ holds. As shown in the conclusion of the proof of Theorem \ref{T1}, this forces the solution $B$ to satisfy
\[[B_{\bullet}B]=0\quad and\quad \na_{A}B=0.\]
Together with the hypothesis $B\not\equiv0$, Lemma \ref{L4} then implies that $\mathrm{rank}(B)=1$ everywhere.	

Since $B$ has constant rank one, we can write  (cf. \cite[Lemma 4.3.25]{DK90})
$$B=\xi\otimes\w,$$
where $\xi\in\Om^{0}(M,\mathfrak{g}_{P})$ is a unit section, i.e $\langle\xi,\xi\rangle=1$, and $\w\in\Om^{2,+}(M)$. The condition $\na_{A}B=0$ becomes
\begin{equation}\label{E13}
0=\na_{A}(\xi\otimes\w)=\na_{A}\xi\wedge\w+\xi\otimes \na\w.
\end{equation}
Taking the inner product with $\xi$ and using $\langle\xi,\xi\rangle=1$, we obtain $\langle\xi,d_{A}\xi\rangle=0$. Hence pairing (\ref{E13}) with $\xi$ gives
\begin{equation*}
0=\langle\na_{A}\xi,\xi\rangle\wedge\w+\langle\xi,\xi\rangle\na\w=\na\w.
\end{equation*}
Thus $\w$ is a non‑zero, parallel self‑dual two-form on $(M,g)$. The existence of a non-degenerate, parallel self-dual two-form $\w$ on a four-manifold is a well-known characterization of a Kähler surface (cf. \cite[Theorem 3.3]{Gur00}). Moreover, the remaining part of (\ref{E13}) then forces 
$$\na_{A}\xi=0,$$ 
which means that the connection $A$ is reducible. 

Conversely, suppose $(M,g)$ is a K\"{a}hler manifold of non-negative scalar curvature. Then by a result of Gursky \cite[Theorem 3.3 (ii)]{Gur00}, the Yamabe constant satisfies $Y(g)=2\sqrt{6}\|W^{+}_{g}\|_{L^{2}}$ (or see \cite{De83}). This proves the equivalence stated in Theorem \ref{T2}.
\end{proof}

We now specialize further to the case where the underlying manifold $M$ is a compact Kähler surface. 
Let $\w$ denote the Kähler form. In a local orthonormal frame, any self-dual two-form $B$ can be decomposed as 
\[ B=\b-\b^{\ast}+B_{0}\w,\]
where $\b\in\Om^{2,0}(\mathfrak{g}_{P})$, $B_{0}\in\Om^{0}(\mathfrak{g}_{P})$. 
Setting $\gamma=-\sqrt{-1}B_{0}$, this allows us to reformulate the Vafa--Witten equations in terms of holomorphic data; 
see \cite[Chapter 7]{Mar10} for a comprehensive discussion.

\begin{theorem}(\cite[Theorem 7.1.2]{Mar10} )\label{T4}
	On a closed K\"{a}hler surface, the Vafa--Witten equations are equivalent to the following system 
	for a triple $(A,\b,\gamma)\in\mathcal{A}_{P}^{1,1}\times\Om^{2,0}(M,\mathfrak{g}_{P})\times\Om^{0}(M,\mathfrak{g}_{P})$:
	\begin{equation}
	\left\{
	\begin{split}
	&\sqrt{-1}\w\wedge F_{A}+\frac{1}{2}[\b\wedge\b^{\ast}]=0,\\
	&\bar{\pa}_{A}\b=d_{A}\gamma=0,\\
	&[\gamma,\gamma^{\ast}]=[\gamma,\b+\b^{\ast}]=0.
	\end{split}
	\right.
	\end{equation}
\end{theorem}

Mares also explored the relationship between the existence of solutions to these equations and the algebro-geometric notion of stability for vector bundles, as detailed in \cite{Mar10,Tan15}. The following proposition shows that on a Kähler surface with positive scalar curvature, any nontrivial Vafa--Witten solution must have $B=0$ if the connection $A$ is irreducible.

\begin{proposition}\label{P5}
	Let $(M,g)$ be a closed K\"{a}hler surface with K\"{a}hler metric $g$. Let $(A,B)$ be a solution of the Vafa--Witten equations on $(M,g)$.
	 Suppose the scalar curvature of $g$ is positive. Then $(A,B)$ satisfies 
	$$F_{A}^{+}=0,\quad \b=0,\quad d_{A}\gamma=0.$$
	Moreover, if $A$ is an irreducible connection, then $B=0$. 
\end{proposition}

\begin{proof}
	Combining the Weitzenböck formula from \cite[Proposition 2.2]{Hua24} (cf. \cite[Proposition 2.3]{Ito87}) with the identities in Theorem \ref{T4},  we obtain the following identity for $\b$:
	\begin{equation*}
	\begin{split}
	0&=\int_{M}\langle\na_{A}^{\ast}\na_{A}\b,\b\rangle+2\int_{M}R_{g}|\b|^{2}+\int_{M}\langle[\sqrt{-1}\La_{\w}F_{A},\b],\b\rangle\\
	&=\|\na_{A}\b\|_{L^2}^{2}+2\int_{M}R_{g}\b^{2}+2\|[\b\wedge\b^{\ast}]\|^{2}_{L^2}.
	\end{split}
	\end{equation*}
	Since the scalar curvature is positive, all terms are non-negative; therefore $\b=0$ and $\La_{\w}F_{A}=0$. 
	
	If in addition $A$ is irreducible, then the equation  $d_{A}\gamma=0$ forces $\gamma=0$. Consequently, we obtain $B=0$. 
\end{proof}

We conclude this section with the application to Einstein manifolds.
\begin{proof}[\textbf{Proof of Theorem \ref{T3}}]
Let $(M,g)$ be a closed, oriented four‑dimensional Einstein manifold with positive scalar curvature, normalized so that $\mathrm{Ric}(g)=3g$. For any such metric, the Chern–Gauss–Bonnet and signature formulas give
\begin{equation}\label{E14}
2\pi^{2}(2\chi(M)+3\sigma(M))=\int_{M}|W^{+}_{g}|^{2}dvol_{g}+\frac{1}{48}\int_{M}|R_{g}|^{2}dvol_{g}.
\end{equation}
By Theorem \ref{T1}, any Vafa–Witten solution $(A,B)$ with $B\not\equiv0$ satisfies
\begin{equation}\label{E15}
\int_{M}|W^{+}|^{2}\geq\frac{4}{3}\pi^{2}(\chi(M)+3\sigma(M)).
\end{equation}
Inserting (\ref{E15}) into (\ref{E14}) and using the normalization $R_{g}=12$, we obtain
\begin{equation*}
\begin{split}
2\pi^{2}(2\chi(M)+3\sigma(M))&\geq\frac{4}{3}\pi^{2}(\chi(M)+3\sigma(M))\\
&+\frac{1}{48}\int_{M}144dvol_{g}.\\
\end{split}
\end{equation*}
Consequently,
\begin{equation}\label{E9}
\frac{2}{9}\pi^{2}(2\chi(M)+3\sigma(M))\geq vol(g).
\end{equation}
Equality in (\ref{E9}) holds if and only if $Y(g)=2\sqrt{6}|W^{+}_{g}|_{L^{2}}$. By Theorem \ref{T2}, this would imply that $(M,g)$ is Kähler and that $A$ is irreducible, contradicting the hypothesis. Hence the inequality is strict. 

Moreover, if $(M,g)$ is K\"{a}hler, then by Proposition \ref{P5}, we would obtain  $B=0$, which again contradicts the assumption $B\not\equiv0$. Therefore, $M$ cannot be K\"{a}hler. This completes the proof of Theorem \ref{T3}.
\end{proof}

\subsection{Stable flat connections on $3$-manifold}\label{sec:dimred} 

In this subsection we establish a correspondence between stable flat connections on a three-manifold $N$ 
and $S^{1}$-invariant solutions of the Vafa--Witten equations on the product four-manifold $M:=S^{1}\times N$. 
Using the estimates developed in the previous sections, we derive geometric constraints on the Yamabe constant of $g_{M}=dt^{2}+g_{N}$ in terms of the Ricci curvature of $g_{N}$.

Let $(N,g_{N})$ be an oriented, closed, smooth $3$-dimensional manifold with smooth Riemannian metric $g$,
 and let $P\rightarrow N$ be a principal $G$-bundle with  compact Lie group $G$. 
 Denote by $\mathcal{A}_{P}$ the space of all connections on $P$, 
 and by $\Om^{k}(N,\mathfrak{g}_{P})$ the space of  $\mathfrak{g}_{P}$-valuled $k$-forms. 
 For a complex connection $\mathcal{A}:=A+\sqrt{-1}\phi$, where $A$ is a connection on $P$ and $\phi\in\Om^{1}(N,\mathfrak{g}_{P})$, 
 the curvature is defined as
\[\mathcal{F}_{\mathcal{A}}=F_{A}-\frac{1}{2}[\phi\wedge\phi]+\sqrt{-1}d_{A}\phi,\]
which is a $2$-form taking values in $P\times_{G}(\mathfrak{g}_{P}^{\C})$.

\begin{definition}(cf. \cite{Cor88,GU12})
	A pair $(A,\phi)$ is called a \textit{stable flat connection}, if it satisfies the elliptic system 
	\begin{equation*}
	\left\{
	\begin{split}
	&F_{A}-\phi\wedge\phi=0,\\
	&d_{A}\phi=d_{A}^{\ast}\phi=0.
	\end{split}
	\right.
	\end{equation*}
\end{definition}
These equations are invariant under both the real gauge group $\mathcal{G}_{P}=C^{\infty}(P\times_{G}G)$ 
and  the complex gauge group $\mathcal{G}^{\C}_{P}:=C^{\infty}(P\times_{G}G_{\C})$. 
On a compact Riemann surface $\Sigma$, solutions to (\ref{E11}) coincide with those of Hitchin's equations (cf. \cite{Hit87}). 
Moreover, by \cite[Proposition 2.2.3]{DK90} (see also \cite[Proposition 1.2.6]{Kob87}), 
the gauge equivalence classes of flat connections on a connected manifold $M$ are in one-to-one correspondence 
with the conjugacy classes of representations $\pi_{1}(M)\rightarrow G$. 

Now return to the setting of this article. Let $N$ be an oriented, smooth, Riemannian three-manifold. Consider the product manifold  $M:=S^{1}\times N$ endowed with the product metric $g_{M}=dt^{2}+g_{N}$. Let
$$p_{1}:S^{1}\times N\rightarrow N$$
be the canonical projection, and denote by $p_{1}^{\ast}(P)\rightarrow M$ the pullback bundle. For a connection $A$ on $P\rightarrow N$, we pull it back to a connection on $p_{1}^{\ast}(P)\rightarrow M$ , still denoted by $A$ by abuse of notation. Define a section $B\in\Ga(M,\Om^{2,+}\otimes\mathfrak{g}_{p^{\ast}_{1}(P)})$ by
\begin{equation}\label{E16}
B=(1+\ast^{M})\ast^{N}p_{1}^{\ast}(\phi),
\end{equation}
where $\ast^{N}$ and $\ast^{M}$ denote the Hodge star operators with respect to $g_{N}$ and $g_{M}$, respectively. The following proposition establishes a one‑to‑one correspondence between stable flat connections on$N$ and $S^{1}$-invariant solutions of the Vafa–Witten equations on $M$.
\begin{proposition}(\cite[Proposition 2.3]{Hua22})\label{P6}
	The canonical projection $p_{1}:S^{1}\times N\rightarrow N$ gives a one-to-one correspondence between stable flat  connections $(A,\phi)$ on $P$ and $S^{1}$-invariant solutions $(A,B)$ of the Vafa-Witten equations on the pullback bundle $p^{\ast}_{1}(P)\rightarrow M$, where $B$ is give by (\ref{E16}).
\end{proposition}

We now apply the main estimate from Theorem \ref{T1} to the product manifold $M=S^{1}\times N$.
\begin{proof}[\textbf{Proof of Theorem \ref{T5}}]
	By Proposition \ref{P6}, a stable flat connection \((A, \phi)\) on \(N\) with \(\phi\not\equiv0\) gives rise to a nontrivial \(S^1\)-invariant solution \((A, B)\) of the Vafa--Witten equations on \(M = S^1 \times N\).  Since the solution is \(S^1\)-invariant and nontrivial, Theorem \ref{T1} applies and yields
	\[
	Y(g_{M}) \leq 2\sqrt{6} \, \| W_{g_{M}}^{+} \|_{L^2(M)}.
	\]
	On the product manifold \(S^1 \times N\) equipped with the product metric $g_{M}=dt^{2}+g_{N}$, 
	the self-dual Weyl tensor \(W_g^+\) can be expressed in terms of the Ricci curvature of the three-dimensional factor. 
	More precisely, one has the pointwise identity (see the Appendix)
	\[
	|W_{g_{M}}^{+}|^{2}= \frac{1}{4}\left| \operatorname{Ric}(g_N) - \frac{1}{3} R_{g_N} g_N \right|^{2}.
	\]
	Integrating over \(M\) and using the \(S^1\)-invariance, we obtain
	\[
	\int_{M} | W_{g_{M}}^{+}|^{2}dvol_{g_{M}}=\frac{1}{2}\mathrm{Vol}(S^{1})\int_{N}\left|\operatorname{Ric}(g_N) - \frac{1}{3} R_{g_N} g_N \right|^{2}dvol_{g_{N}}.
	\]
	Substituting this into the previous inequality yields the desired bound on the Yamabe constant. 
	
	In particular, if $g_{N}$ is Einstein, then $\operatorname{Ric}(g_N) =\frac{1}{3}R_{g_N} g_N $. Consequently, $Y(g_{M})\leq 0$.
	 This completes the proof of Theorem \ref{T5}. 
\end{proof}

\section{Eigenvalue Gap and Compactness of Vafa-Witten Solutions}

The operator $d_{A}^{+}d^{+,\ast}_{A}$ is an elliptic self-adjoint operator on the space of sections of $\Om^{2,+}(M,\mathfrak{g}_{P})$. 
A standard result states that its eigenvalues are discrete and the smallest eigenvalue is nonnegative.
\begin{definition}
Let $(M,g)$ be a closed, four-dimensional Riemannian
manifold and $P\rightarrow M$ be a principal $G$-bundle with $G$ being a compact Lie group. Let $A$ be a smooth connection on $P$. 
The least eigenvalue of  $d_{A}^{+}d_{A}^{+,\ast}$ is defined as
\end{definition}
\begin{equation*}
\mu(A):=\inf_{v\in\Om^{2,+}(M,\mathfrak{g}_{P})\backslash\{0\}}\frac{\|d^{+,\ast}_{A}v\|_{L^2}^{2}}{\|v\|_{L^2}^{2}}.
\end{equation*}
An anti-self dual(ASD) connection $A_{asd}$ is  called \textit{regular} if 
$$\ker d^{+,\ast}_{A_{asd}}\cap\Om^{2,+}(M,\mathfrak{g}_{P})=\{0\},$$ 
i.e. $\mu(A_{asd})>0$.

We now recall a key a compactness theorem due to Sedlacek.

\begin{theorem}(cf. \cite[Theorem 4.3]{Sed82})\label{T7}
Let $(M,g)$ be a closed, four-dimensional Riemannian manifold and $P\rightarrow M$ be a principal $G$-bundle with $G$ being a compact Lie group. If $\{A_{i}\}_{i\in\mathbb{N}}$ is a sequence $C^{\infty}$ connection on $P$ such that $\{F^{+}_{A_{i}}\}_{i\in\N}$ converges to zero in $L^{2}$-topology, then there exists\\
(1) an integer $L$ and a finite set of points, $\Sigma=\{x_{1},\ldots,x_{L}\}\subset M$, $\Sigma$ maybe empty;\\
(2) a smooth anti-self-dual $A_{\infty}$ on a principal $G$-bundle $P_{\infty}$ over $M$ with $\eta(P_{\infty})=\eta(P)$;\\
(3) a subsequence $\Xi\subset\N$, we also denote by $\{A_{i}\}_{i\in\Xi}$, a sequence gauge transformation $\{g_{i}\}_{i\in\Xi}$ such that, $g^{\ast}_{i}(A_{i})$ weakly converges to $A_{\infty}$ in $L^{2}_{1}$ on $M\backslash\Sigma$, and $g_{i}^{\ast}(F_{A_{i}})$ weakly converges to $F_{A_{\infty}}$ in $L^{2}$ on $M\backslash\Sigma$.
\end{theorem}	

Feehan established that the least eigenvalue of the operator $d_{A}^{+}d_{A}^{+,\ast}$ is locally $L^{p}$-continuous with respect to the connection for all $2\leq p\leq 4$ (cf. \cite[Proposition A.3]{Fee16}). 

\begin{proposition}\label{P7}
	Let $(M,g)$ be a closed, four-dimensional Riemannian manifold and $P\rightarrow M$ be a principal $G$-bundle with $G$ being a compact Lie group.
	 If $\{A_{i}\}_{i\in\mathbb{N}}$ is a sequence $C^{\infty}$ connection on $P$ such that $\{F^{+}_{A_{i}}\}_{i\in\N}$ converges to zero in $L^{2}$-topology,
	  then
$$\lim_{i\rightarrow\infty}\mu(A_{i})=\mu(A_{\infty}),$$
where $A_{\infty}$ is the limit connection given by Theorem \ref{T7}.
\end{proposition} 

As an immediate consequence of Theorem \ref{T7} and Proposition \ref{P7}, we obtain the following result for sequences of Vafa--Witten solutions.

\begin{corollary}\label{C3}
Let $(M,g)$ be a closed, four-dimensional Riemannian manifold and $P\rightarrow M$ be a principal $G$-bundle with $G$ being a compact Lie group. Let $\{(A_{i},B_{i})\}_{i\in\mathbb{N}}$ be a sequence $C^{\infty}$ solutions of Vafa--Witten equations with $F_{A_{i}}^{+}\not\equiv 0$. Suppose  $\{F^{+}_{A_{i}}\}_{i\in\N}$ converges to zero in $L^{2}$-topology, then there exists a smooth anti-self-dual $A_{\infty}$ on a principal $G$-bundle $P_{\infty}$ over $M$ with $\eta(P_{\infty})=\eta(P)$ such that $\mu(A_{\infty})=0$.
\end{corollary}
\begin{proof}
Since $F_{A}^{+}=-\frac{1}{8}[B_{\bullet} B]$, the condition $F_{A}\not\equiv0$ implies $B\not\equiv0$. 
The equation $d_{A}^{+,\ast}B=0$ implies that $\mu(A_{i})=0$. 
By Theorem \ref{T7} and Proposition \ref{P7}, we get $$\mu(A_{\infty})=\lim_{i\rightarrow\infty}\mu(A_{i})=0.$$
\end{proof}
\begin{remark}
Even if $\|F^{+}_{A_{i}}\|_{L^{2}}$ tends to zero, the $L^2$-norm of the field $B_{i}$ may still tend to infinity. When $\|B_{i}\|_{L^{2}}$ becomes unbounded, the corresponding compactification process---studied by Taubes~\cite{Tau17}---is more complicated than the Uhlenbeck compactification. Corollary \ref{C3} shows that, in the sense of Uhlenbeck compactification, the limiting connection $A_{\infty}$ must be a non‑regular ASD connection.
\end{remark}
Now, we are ready to prove Theorem \ref{T8}.
\begin{proof}[\textbf{Proof of Theorem \ref{T8}}]
Suppose, for contradiction, that no such constant $\varepsilon$ exists. 
Then we may choose a sequence solutions $\{(A_{i},B_{i})\}_{i\in\N}$ of Vafa--Witten equations such that 
$$\|F_{A_{i}}^{+}\|_{L^{2}(X)}\rightarrow 0\quad as \quad i\rightarrow\infty.$$ 
According to Corollary \ref{C3}, $A_{i}$ converges to an ASD connection $A_{\infty}$ in the sense of $L^{2}_{1,loc}(M)$.
 Then $\lim_{i\rightarrow\infty}\mu(A_{i})=\mu(A_{\infty})=0$,
  contradicting our initial assumption regarding the ASD connection $A_{\infty}$ in $\overline{\mathcal{M}}_{ASD}(P, g)$ is regular,
   i.e. $\mu(A_{\infty})>0$. Hence such $\varepsilon$ must exist.
\end{proof}

We denote 
$$\mathfrak{B}_{\varepsilon}(P,g):=\{[A]\in\mathcal{A}_{P} :\|F_{A}^{+}\|_{L^{2}(M)}<\varepsilon \}.$$
By combining the continuity of the eigenvalue with Sedlacek's compactness theorem, Feehan established that, provided every ASD connection in  $\overline{\mathcal{M}}_{ASD}(P, g)$ is regular, there exists a constant $\varepsilon>0$ such that the eigenvalue $\mu(A)$  is bounded below by a uniform positive constant for all connections $A\in\mathfrak{B}_{\varepsilon}(P,g)$.

\begin{theorem}\label{T6}
Let $(M,g)$ be a closed, four-dimensional Riemannian manifold and $P\rightarrow M$ be a principal $G$-bundle with $G$ being a compact Lie group.
 Suppose that all connections in $\overline{\mathcal{M}}_{ASD}(P, g)$ are \textit{regular}.  
 Then there exist  positive constant $\mu_{0}=\mu_{0}(g,P)$ and $\varepsilon=\varepsilon(g,P)$ such that 
\begin{equation*}
\begin{split}
&\mu(A)\geq\mu_{0},\quad \forall ~ [A] \in \overline{\mathcal{M}}_{ASD}(P, g),\\
&\mu(A)\geq\frac{\mu_{0}}{2},\quad \forall ~ [A] \in\mathfrak{B}_{\varepsilon}(P,g).\\
\end{split}
\end{equation*}
\end{theorem}

\begin{remark}
The Theorem \ref{T6} is due to Feehan \cite[Theorems 3.7 and 3.8]{Fee16}. Using his result, we obtain an alternative proof of Theorem \ref{T8} as follows.
\end{remark}

\begin{proof}[\textbf{An alternative proof of Theorem \ref{T8}}]
Suppose, to the contrary, that there exists a smooth solution $(A,B)$ of the Vafa--Witten equations
 such that  $F_{A}^{+} \not\equiv 0$ and $\|F_A^+\|_{L^2(M)} < \varepsilon(g,P)$.
  Since $F_{A}^{+}=-\frac{1}{8}[B_{\bullet} B]$, the condition $F_{A}\not\equiv0$ implies $B\not\equiv0$. 

By Theorem \ref{T6}, we have $\mu(A) \ge \frac{\mu_0}{2} > 0$.  From  the Vafa--Witten equation $d_{A}^{+,\ast}B = 0$, it follows that
\[
0 = \|d_A^{+,\ast} B\|_{L^2(M)}^2\geq \mu(A) \|B\|_{L^2(M)}^2\geq\frac{\mu_{0}}{2}\|B\|_{L^2(M)}^2.
\]
Consequently, $\|B\|_{L^2(M)} = 0$, i.e. $B \equiv 0$, contradicting the assumption $B \not\equiv 0$. Therefore any smooth solution must satisfy $\|F_A^+\|_{L^2(M)} \ge \varepsilon(g,P)$ unless $F_A^+ \equiv 0$.
\end{proof}

\section*{Acknowledgements}
This work was supported by the National Natural Science Foundation of China (Grant Nos. 12271496 to Huang and 12201001 to Zhang), 
the Youth Innovation Promotion Association of the Chinese Academy of Sciences.
 The authors also thank DeepSeek for its assistance in proofreading and improving the grammar and expression of this manuscript.

\section*{Declarations}

\noindent\textbf{Data availability} {This manuscript has no associated data.}

\noindent\textbf{Conflict of interest} The authors state that there is no conflict of interest.

\section*{Appendix: Computation of $|W^+|^2$ for $M = S^1 \times N$}\label{app:compute}
\label{app:Wplus}
The Riemann curvature tensor admits a decomposition into its Weyl tensor $W_{g}$ and a term involving the Schouten tensor:
\begin{equation}
Rm=S_{g}\odot g+W_{g},
\end{equation}
where $\odot$ denotes the Kulkarni--Nomizu product. The Schouten tensor of $g$, denote by $S_{g}$, is given by
\begin{equation}
S_{g}=\frac{1}{n-2}(\mathrm{Ric}_{g}-\frac{R_{g}}{2(n-1)}\cdot g),
\end{equation}
where $\mathrm{Ric}_{g}$ and $R_{g}$ are the Ricci tensor and scalar curvature of $g$ respectively.

Let $(N,g_N)$ be a closed $3$-manifold and consider the product manifold $M = S^1 \times N$ with metric $g_M = dt^2 + g_N$. We work in a local orthonormal frame $\{e_0,e_1,e_2,e_3\}$ where $e_0 = \partial_t$ and $\{e_i\}_{i=1}^3$ is an orthonormal frame on $N$.

For a product metric, the all Riemann curvature components are:
\[
R_{ijkl} = R_{ijkl}^{(N)}, \qquad R_{0i0j} =0,\qquad R_{0ijk}=0,
\]
where $R_{ijkl}^{(N)}$ denotes the Riemann tensor of $g_N$. Consequently,
\[
R_{00} = 0,\quad R_{0i}=0,\quad R_{ij}=R_{ij}^{(N)},\quad R_M = R_N.
\]
In four dimensions, the Weyl tensor $W$ is given by
\[
W_{abcd} = R_{abcd} - \bigl(g_{ac}S_{bd}+g_{bd}S_{ac}-g_{ad}S_{bc}-g_{bc}S_{ad}\bigr),
\]
where $S$ is the Schouten tensor
\[
S_{ab} = \frac{1}{2}\Bigl(R_{ab} - \frac{S_M}{6}g_{ab}\Bigr).
\]
Using the product curvature identities, we compute:
\[
S_{00} = -\frac{S_N}{12},\qquad S_{0i}=0,\qquad S_{ij} = \frac12\Bigl(R_{ij}^{(N)}-\frac{S_N}{6}g_{ij}\Bigr).
\]
Define the trace-free Ricci tensor of $N$:
\[
\operatorname{Ric}_0^{(N)} = \operatorname{Ric}^{(N)} - \frac{S_N}{3}g_N.
\]
Then $R_{ij}^{(N)} = (\operatorname{Ric}_0^{(N)})_{ij} + \frac{S_N}{3}g_{ij}$, and substituting gives
\[
S_{ij} = \frac12 (\operatorname{Ric}_0^{(N)})_{ij} + \frac{S_N}{12}g_{ij}.
\]
A direct computation using the definition of the Weyl tensor gives the following components:
\begin{equation}\label{EA1}
\begin{split}
&W_{0i0j} = -S_{ij} - g_{ij}S_{00} = -\frac12 (\operatorname{Ric}_0^{(N)})_{ij},\\
&W_{0ijl} = 0,\\
&W_{ijkl}=\frac{1}{2}(\operatorname{Ric}_0^{(N)}\odot g)_{ijkl}
\end{split}
\end{equation}
The space of self-dual $2$-forms $\Lambda^+$ is spanned by
\[
\Sigma_1^+ = e_0\wedge e_1 + e_2\wedge e_3,\;
\Sigma_2^+ = e_0\wedge e_2 + e_3\wedge e_1,\;
\Sigma_3^+ = e_0\wedge e_3 + e_1\wedge e_2.
\]
The self-dual Weyl tensor $W^+$ is the restriction of $W$ to $\Lambda^+$. It is known that for a product metric $dt^2 + g_N$ one has
\begin{equation}\label{EA2}
|W^+|^2 = \frac{1}{4} |\operatorname{Ric}_0^{(N)}|^2.
\end{equation}
A concise derivation proceeds as follows. At a point where $\operatorname{Ric}_0^{(N)}$ is diagonal with eigenvalues $\lambda_1,\lambda_2,\lambda_3$, so  that 
$$\lambda_1+\lambda_2+\lambda_3=0.$$ 
One computes using (\ref{EA1}):
\[
W(\Sigma_i^+, \Sigma_i^+) =-\lambda_i,\qquad W(\Sigma_i^+, \Sigma_j^+) = 0\;(i\neq j).
\]
With the normalization  $\langle \Sigma_i^+,\Sigma_j^+\rangle = 2\delta_{ij}$, the endomorphism $W^+$ has matrix entries
$$(W^+)_{ij} = \frac{1}{2} W(\Sigma_i^+,\Sigma_j^+),$$
giving $(W^+)_{ii} = -\frac{1}{2}\lambda_i$. Hence
\[
|W^+|^2 = \sum_{i,j} (W^+_{ij})^2 = \sum_i \bigl(-\frac{1}{2}\lambda_i\bigr)^2 = \frac{1}{4}\sum_i \lambda_i^2.
\]
Since $|\operatorname{Ric}_0^{(N)}|^2 = \sum_i \lambda_i^2$, we obtain (\ref{E2})
This identity is used in the proof of Theorem \ref{T5}.

\end{document}